\documentclass{amsart}
\usepackage{etex}
\usepackage[style=alphabetic,doi=false,isbn=false,url=false,maxbibnames=99,backend=bibtex]{biblatex}
\DeclareFieldFormat[article,incollection,unpublished]{title}{#1}\renewbibmacro{in:}{\ifentrytype{article}{}{\printtext{\bibstring{in}\intitlepunct}}}

\usepackage{diagrams}
\newarrow{Equals}{=}{=}{=}{=}{=}
\newarrow{Spectral}{=}{=}{=}{=}{>}
\newarrow{dotsto}{.}{.}{.}{.}{>}

\def\dlog{\mathrm{dlog}}

\usepackage{fixltx2e}
\usepackage{tikz}
\usepackage{tikz-cd}
\usepackage{adjustbox}
\usepackage{enumitem}
\usepackage{dsfont}

\newcommand{\Tr}{\operatorname{Tr}}
\newcommand{\cyc}{\operatorname{cyc}}
\newcommand{\triv}{\operatorname{triv}}

\newcommand{\LBDJ}{L_{\mathrm{BDJ}}}
\newcommand{\LDDR}{L_{\mathrm{DDR}}}

\bibliography{explicitBDJ}

\usepackage{amssymb,amsmath,amsfonts,amsthm,epsfig,amscd}
\usepackage{stmaryrd}
\usepackage[all,cmtip,poly]{xy}
\usepackage{color}

\newcommand{\pr}{\operatorname{pr}}

\def\iso{\buildrel \sim \over \longrightarrow}

\newcommand{\id}{\operatorname{id}}

\newtheorem{thm}[subsubsection]{Theorem}
\newtheorem{lemma}[subsubsection]{Lemma}

\newtheorem{lem}[subsubsection]{Lemma}

\newtheorem{cor}[subsubsection]{Corollary}

\newtheorem{prop}[subsubsection]{Proposition}

\theoremstyle{definition}
\newtheorem{df}[subsubsection]{Definition}

\theoremstyle{remark}

\newtheorem{rem}[subsubsection]{Remark}

\def\numequation{\addtocounter{subsubsection}{1}\begin{equation}}
\def\nummultline{\addtocounter{subsubsection}{1}\begin{multline}}
\def\anumequation{\addtocounter{subsection}{1}\begin{equation}}
\def\anummultline{\addtocounter{subsection}{1}\begin{multline}}

\newif\iffinalrun
\iffinalrun
\else
 \fi

\iffinalrun
  \newcommand{\need}[1]{}
  \newcommand{\mar}[1]{}
\else
  \newcommand{\need}[1]{{\tiny *** #1}}
  \newcommand{\mar}[1]{\marginpar{\raggedright\tiny #1}}
\fi

\newcommand{\F}{\FF}

\newcommand{\Q}{\QQ}

\newcommand{\Z}{\ZZ}

\newcommand{\m}{\frakm}

\newcommand{\FF}{{\mathbb F}}

\newcommand{\QQ}{{\mathbb Q}}

\newcommand{\ZZ}{{\mathbb Z}}

\newcommand{\cF}{{\mathcal F}}

\newcommand{\cM}{{\mathcal M}}

\newcommand{\cO}{{\mathcal O}}

\newcommand{\frakm}{\mathfrak{m}}

\newcommand{\Fbar}{\overline{\F}}
\newcommand{\Qbar}{\overline{\Q}}

\newcommand{\Fp}{\F_p}

\newcommand{\Fpbar}{\Fbar_p}

\newcommand{\Fpbartimes}{\Fpbar^{\times}}

\newcommand{\Zp}{\Z_p}

\newcommand{\Qp}{\Q_p}

\newcommand{\Qpbar}{\Qbar_p}

\DeclareMathOperator{\Gal}{Gal}
\DeclareMathOperator{\GL}{GL}

\DeclareMathOperator{\Hom}{Hom}

\newcommand{\Frob}{\mathrm{Frob}}

\newcommand{\sep}{\mathrm{sep}}

\newcommand{\rhobar}{\overline{\rho}}

                               \newcommand{\into}{\hookrightarrow}

\newcommand{\Art}{{\operatorname{Art}}}

\newcommand{\epsilonbar}{\overline{\epsilon}}

\newcommand{\Res}{\operatorname{Res}}

\usepackage[bookmarksopen,bookmarksdepth=2]{hyperref}

\begin{document}
\title{Explicit Serre weights for two-dimensional Galois representations}

\author[F. Calegari]{Frank Calegari}\email{fcale@math.uchicago.edu}
\address{Department of Mathematics, University of Chicago,
5734 S.\ University Ave., Chicago, IL 60637, USA}

\author[M. Emerton]{Matthew Emerton}\email{emerton@math.uchicago.edu}
\address{Department of Mathematics, University of Chicago,
5734 S.\ University Ave., Chicago, IL 60637, USA}

\author[T. Gee]{Toby Gee} \email{toby.gee@imperial.ac.uk} \address{Department of
  Mathematics, Imperial College London,
  London SW7 2AZ, UK}
\author[L. Mavrides]{Lambros Mavrides}
\email{lambros.mavrides@kcl.ac.uk} \address{Department of
  Mathematics, King's College London, London WC2R 2LS, UK}

\thanks{The first author is supported in part by NSF  Grant
  DMS-1648702. The second author is supported in part by the
  NSF grant DMS-1303450. The third author is 
  supported in part by a Leverhulme Prize, EPSRC grant EP/L025485/1, Marie Curie Career
  Integration Grant 303605, and by
  ERC Starting Grant 306326. }

\maketitle
\begin{abstract}
  We prove the explicit version of the Buzzard--Diamond--Jarvis
  conjecture formulated in~\cite{2016arXiv160307708D}. More precisely,
  we prove that it is equivalent to the original
  Buzzard--Diamond--Jarvis conjecture, which was proved for odd primes
  (under a mild Taylor--Wiles hypothesis) in earlier work of the third
  author and coauthors.
\end{abstract}
\setcounter{tocdepth}{1}
\tableofcontents
\section{Introduction}The weight part of Serre's conjecture
for 
Hilbert modular forms predicts the weights of the Hilbert modular
forms giving rise to a particular modular mod~$p$ Galois
representation, in terms of the restrictions of this Galois
representation to decomposition groups above~$p$. The conjecture was
originally formulated in~\cite{bdj} in the case that~$p$ is unramified in the
totally real field. Under a mild Taylor--Wiles hypothesis on the image
of the global Galois representation, this conjecture has been proved
for~$p>2$ in a series of papers of the third author and coauthors,
culminating in the paper~\cite{MR3324938}, which proves a
generalization allowing~$p$ to be arbitrarily ramified. We refer the
reader to the introduction to~\cite{MR3324938} for a discussion of
these results.

Let $K/\Qp$ be an unramified extension, and
let~$\rhobar:G_K\to\GL_2(\Fpbar)$ be a (continuous)
representation. If~$\rhobar$ is irreducible, then the recipe for
predicted weights in~\cite{bdj} is completely explicit, but in the
case that it is a non-split extension of characters, the recipe is in
terms of the reduction modulo~$p$ of certain crystalline extensions of
characters. This description is not useful for practical computations, and the
recent paper~\cite{2016arXiv160307708D} proposed an alternative recipe
in terms of local class field theory,
along with the Artin--Hasse exponential,
which can be made completely explicit in
concrete examples. (Indeed, \cite[\S9-10]{2016arXiv160307708D} gives
substantial numerical evidence for their conjecture.)

In this paper, we prove~\cite[Conj.\ 7.2]{2016arXiv160307708D}, which
says that the recipes of~\cite{bdj} and~\cite{2016arXiv160307708D}
agree. This is a purely local conjecture, and our proof is purely
local. Our main input is the results of~\cite{MR3164985} (and their
generalization to~$p=2$ in~\cite{Wangp2}). We briefly sketch our
approach. Suppose that~$\rhobar\cong
\begin{pmatrix}
  \chi_1&*\\0&\chi_2
\end{pmatrix}
$, and set~$\chi=\chi_1\chi_2^{-1}$. For a given Serre weight, the
recipes of~\cite{bdj} and~\cite{2016arXiv160307708D} determine
subspaces~$\LBDJ$ and~$\LDDR$ of $H^1(G_K,\chi)$, and we have to
prove that~$\LBDJ=\LDDR$.

Let~$K_\infty/K$ be the (non-Galois) extension obtained by adjoining a compatible
system of $p^n$th roots of a fixed uniformizer of~$K$ for all~$n$. The
restriction map $H^1(G_K,\chi)\to H^1(G_{K_\infty},\chi)$ is injective
unless~$\chi$ is the mod~$p$ cyclotomic character, and~\cite[Thm.\
7.9]{MR3164985} allows us to give an explicit description of the image
of~$\LBDJ$ in~$H^1(G_{K_\infty},\chi)$ in terms of Kisin modules. The
theory of the field of norms gives a natural isomorphism
of~$G_{K_\infty}$ with~$G_{k((u))}$, where~$k$ is the residue field
of~$K$, and we obtain a description of the image of~$\LBDJ$
in~$H^1(G_{k((u))},\chi)$ in terms of Artin--Schreier theory. On the other hand, we prove a compatibility of the Artin--Hasse
exponential with the field of norms construction that allows us to compute the
image of $\LDDR$ in~$H^1(G_{k((u))},\chi)$. We then use an explicit
reciprocity law of Schmid~\cite{MR1581502} to reduce the comparison
of~$\LBDJ$ and~$\LDDR$ to a purely combinatorial problem, which we
solve.

It is possible that the conjecture of~\cite{2016arXiv160307708D} could
be extended to the case that~$p$ ramifies in~$K$; we have not tried to
do this, but we expect that if such a generalization exists, it could
be proved by the methods of this paper, using the results
of~\cite{MR3324938}. 

The fourth author's PhD thesis~\cite{Mavrides} proved~\cite[Conj.\
7.2]{2016arXiv160307708D} in generic cases using similar techniques to
those of this paper in the setting of $(\varphi,\Gamma)$-modules
(using the results of~\cite{MR2776609} where we appeal
to~\cite{MR3164985}), while the first three authors arrived separately
at the strategy presented here for resolving the general case.

\subsection{Acknowledgements} We would all like to thank Fred Diamond
for many helpful conversations over the last decade about the weight
part of Serre's conjecture, and for his key role in formulating the
various generalisations of it to Hilbert modular forms. The fourth
author is particularly grateful for his support and guidance over the
course of his graduate studies. We would also like to thank him for
suggesting that the four of us write this joint paper. We are grateful
to Ehud de Shalit for sending us a copy of his paper~\cite{MR1196529},
which led us to Schmid's explicit reciprocity law. We would also like
to thank David Savitt for informing us of the forthcoming
paper~\cite{Wangp2} of Xiyuan Wang, and Victor Abrashkin for informing
us of his paper~\cite{MR1446195}, which extends the results
of~\cite{MR1196529}, and much as in the present paper relates these
results to the Artin--Hasse exponential.
Finally, we would like to thank the anonymous referee for their careful
reading of the paper. 

\section{Notation}\label{sec:notation}We follow the conventions of~\cite{MR3324938}, which are the same as
those in the arXiv version of~\cite{MR3164985} (see~\cite[App.\
A]{MR3324938} for a correction to some of the indices in the published
version of~\cite{MR3164985}). Let~$p$ be prime, and let $K/\Qp$ be a finite
unramified extension of degree~$f$, with residue field $k$. Embeddings
$\sigma:k\into \Fpbar$ biject with $\Qp$-linear embeddings $K\into\Qpbar$, and we
choose one such embedding $\sigma_0:k\into\Fpbar$, and recursively require
that~$\sigma_{i+1}^p=\sigma_i$. Note
that~$\sigma_{i+f}=\sigma_i$. Note also that this convention is
opposite to that of~\cite{2016arXiv160307708D}, so that
their~$\sigma_i$ is our~$\sigma_{-i}$; consequently, to compare our
formulae to those of~\cite{2016arXiv160307708D}, one has to negate the
indices throughout.

If~$\pi$ is a root of $x^{p^f-1}+p=0$ then we have the fundamental
character~$\omega_f:G_K\to k^\times$ defined
by \[\omega_f(g)=g(\pi)/\pi\pmod{\pi\cO_{K(\pi)}}.\] The composite
of~$\omega_f$ with the Artin map~$\Art_K$ (which we normalize so that a
uniformizer corresponds to a geometric Frobenius element) is the
homomorphism $K^\times\to k^\times$ sending~$p$ to~$1$ and sending
elements of $\cO_{K}^\times$ to their reductions modulo~$p$. For
each~$\sigma:k\into\Fpbar$, we
set~$\omega_\sigma:=\sigma\circ\omega|_{I_K}$, and
$\omega_i:=\omega_{\sigma_i}$, so that in particular we
have~$\omega_{i+1}^p=\omega_i$.

If $l/k$ is a finite extension, we choose an
embedding~$\widetilde{\sigma}_0:l\into\Fpbar$ extending~$\sigma_0$, and again set $\widetilde{\sigma}_i=\widetilde{\sigma}_{i+1}^p$.
We have an isomorphism
\numequation
\label{eqn:product over embeddings}
l\otimes_{\Fp} \Fpbar \iso \prod_{\widetilde{\sigma}_i} 
\Fbar_p,
\end{equation}
with the projection onto the factor labelled by $\widetilde{\sigma}_i$
being given by $x\otimes y \mapsto \widetilde{\sigma}_i(x)y.$
Under this isomorphism,
the automorphism $\varphi \otimes \id$ on $l\otimes_{\Fp} \Fpbar$
becomes identified with the automorphism on $\prod \Fbar_p$ given by
$(y_i) \mapsto (y_{i-1}).$

If~$\cM$ is an $l \otimes_{\Fp}\Fpbar$-module equipped with a $\varphi$-linear
endomorphism~$\varphi$, then the isomorphism~(\ref{eqn:product over embeddings})
induces a corresponding decomposition $\cM \iso \prod_i \cM_i,$
and the endomorphism $\varphi$ of $\cM$ induces $\Fbar_p$-linear morphisms
$\varphi:\cM_{i-1}\to \cM_i$.

\section{Results}
\subsection{Fields of norms}We briefly recall (following~\cite[\S1.1.12]{MR2600871}) the theory of the field of norms and of \'etale
$\varphi$-modules, adapted to the case at hand. For each~$n$, let $(-p)^{1/p^n}$ be a
choice of $p^n$-th root of~$-p$, chosen so
that~$((-p)^{1/p^{n+1}})^p=(-p)^{1/p^n}$, and let~$K_n=K((-p)^{1/p^n})$. Write
$K_\infty=\cup_{n}K_n$. Then by the theory of the field of
norms, \[\varprojlim_{N_{K_{n+1}/K_n}}K_n\] (the transition maps being
the norm maps) can be identified with~$k((u))$, with $((-p)^{1/p^n})_n$
corresponding to~$u$.
If $F$ is a finite extension of $K$ (inside some given algebraic closure
of $K$ containing $K_{\infty}$) then $F_{\infty} := F K_{\infty}$ is a finite
extension of $K_{\infty}$, and applying the field of norms construction
to $F_{\infty}$, we obtain a finite separable extension
 \[\cF:=\varprojlim_{N_{FK_n/FK_{n-1}}}FK_n,\] 
of $k((u))$.
If $F$ is Galois over $K$, then $F_{\infty}$ is Galois over $K_{\infty}$,
and also $\cF$ is Galois over $k((u))$,
and there is a natural isomorphism of Galois groups
\numequation
\label{eqn:Galois iso for finite extensions}
\Gal\bigl(\cF/k((u))\bigr) \iso \Gal(F_{\infty}/K_{\infty}),
\end{equation}
and, composing with the canonical homomorphism
$\Gal(F_{\infty}/K_{\infty}) \to \Gal(F/K),$
a natural homomorphism of Galois groups
\numequation
\label{eqn:Galois hom} 
\Gal\bigl(\cF/k((u))\bigr) \to \Gal(F/K).
\end{equation}

Every finite extension of $K_{\infty}$ arises as such an $F_{\infty}$,
and in this manner we obtain a functorial bijection between
finite extensions of $K_{\infty}$ and finite separable extensions
$\cF$ of $k((u))$.
In particular, the various isomorphisms~(\ref{eqn:Galois iso for finite 
	extensions}) piece together to induce a natural isomorphism
of absolute Galois groups
\numequation
\label{eqn:absolute Galois iso}
G_{K_\infty}=G_{k((u))}. 
\end{equation}

The utility of the isomorphism~(\ref{eqn:absolute Galois iso}) arises from
the fact that
there is an equivalence of abelian categories between the category of
finite-dimensional $\Fpbar$-representations $V$ of $G_{k((u))}$ and
the category of \emph{\'etale $\varphi$-modules}. The latter are by
definition finite $k((u))\otimes_{\Fp} \Fpbar$-modules $\mathcal{M}$
equipped with a $\varphi$-semilinear map $\varphi:\cM\to\cM$, with the
property that the induced  $k((u))\otimes_{\Fp} \Fpbar$-linear map
$\varphi^*\cM\to\cM$ is an isomorphism. This equivalence of categories
preserves lengths in the obvious sense, and is given by the
functors\[T:\cM\to(k((u)))^\sep\otimes_{k((u))}\cM)^{\varphi=1}\](where
$k((u))^{\sep}$ is a separable closure of~$k((u))$) and \[V\mapsto (k((u))^\sep\otimes_{\Fp}V)^{G_{k((u))}}.\]
The isomorphism~(\ref{eqn:absolute Galois iso}) then allows us to describe
finite-dimensional representations of $G_{K_{\infty}}$ 
over $\Fbar_p$
via \'etale $\varphi$-modules.  In the subsection~\ref{subsec:
  etale phi modules from GLS}
we make this description completely explicit in the context
of (the restriction to $K_{\infty}$ of) the crystalline extensions of characters
that arise in the conjecture of~\cite{bdj}.

The above isomorphisms of Galois groups are compatible with local
class field theory in a natural way.
Namely, if $F/K$ and $\cF/k((u))$ are as above,
  then the projection map $k((u))=\varprojlim_{N_{K_{n+1}/K_n}}K_n\to K$
  induces a natural map
  \numequation
  \label{eqn:map on norms}
k((u))^\times/N_{\cF/k((u))^\times}\cF^\times\to
K^\times/N_{F/K}F^\times,
\end{equation}
and we have the following result.
\begin{lem}
  \label{lem: class field theory field of norms}If~$F/K$ is a finite
  abelian extension,  then the following
diagram commutes:
\[\xymatrix{\Gal(\cF/k((u)))\ar^-{\mathrm{(\ref{eqn:Galois hom})}}
		[d]\ar[rr]^{\Art^{-1}_{k((u))}} & &
    k((u))^\times/N_{\cF/k((u))^\times}\cF^\times\ar^-{\mathrm{(\ref{eqn:map on norms})}}[d]\\ \Gal(F/K)\ar[rr]^{\Art^{-1}_K} & & K^\times/N_{F/K}F^\times}  \]
\end{lem}
\begin{proof}
  This is easily checked directly, and is a special case
  of~\cite[Prop.\ 5.2]{MR2914899}, which proves a generalization to
  higher-dimensional local fields; see also~\cite{MR951749}, where the
  analogous result is proved for general APF extensions (strictly
  speaking, the result of~\cite{MR951749} does not apply as written in
  our situation, as the extension $K_\infty/K$ is not Galois; but in
  fact the argument still works). In brief, it is enough to check
  separately the cases that~$F/K$ is either unramified or totally
  ramified; in the former case the result is immediate, while the
  latter case follows from Dwork's description of Artin's reciprocity
  map for totally ramified abelian extensions, \cite[XIII \S5 Cor.\ to
  Thm.\ 2]{MR554237}.
\end{proof}

\subsection{Compatibility of pairings}\label{subsec: compatibility of pairings}
It will be convenient to establish a further compatibility between various natural pairings.
For a field~$M$, let~$M^{(p)}/M$ denote the maximal exponent~$p$
abelian extension (inside some fixed algebraic closure). 
If~$M_{\infty}/M$ is an extension, then we have a diagram
as follows (where~$\pr$ is the natural map given by restriction of
automorphisms of~$M_\infty^{(p)}$ to~$M^{(p)}$):
$$
\begin{diagram}
\Gal(M^{(p)}_{\infty}/M_{\infty}) \   &   \times &  \ H^1(G_{M_{\infty}},\Fpbar)  & \rTo & \Fpbar \\
\dTo_{\pr} & & \uTo^{\iota} & & \\
\Gal(M^{(p)}/M)  \   & \times &  \ H^1(G_M,\Fpbar) & \rTo &  \Fpbar. \\
\end{diagram}
$$
\begin{lemma}\label{lem: pairing compatibility} The diagram commutes, in the sense that~$\langle \pr \alpha, \beta \rangle = \langle \alpha, \iota \beta \rangle$.
\end{lemma}

\begin{proof} Since~$H^1(G_M,\Fpbar)=\Hom(G_M,\Fpbar)$ (and similarly
  for~$M_\infty$), since the pairings are given by evaluation, and since~$\iota$
  is the natural restriction map, this is clear.
\end{proof}

Suppose now that~$M$ is a finite extension of~$\Qp$ with residue
field~$l$, and that 
~$\pi$ is a uniformizer of~$M$. If
~$M_{\infty}/M$ is the extension given by a compatible choice
of~$p$-power roots  of~$\pi$, then
$$\Gal(M^{(p)}_{\infty}/M_{\infty}) \simeq  l((u))^{\times} \otimes \Fp$$
via the field of norms construction together with local class field theory
(applied to~$l((u))$). 

On the other hand,
taking Galois cohomology of the short exact sequence
$$0 \to \Fbar_p \to l((u))^{\sep} \otimes_{\Fp} \Fpbar
\buildrel \psi\otimes \id
\over \longrightarrow l((u))^{\sep} \otimes_{\Fp} \Fpbar \to 0,$$
where $\psi:l((u))^{\sep} \to l((u))^{\sep}$
is the Artin--Schreier map defined by $\psi(x) = x^p - x$,
yields 
an isomorphism
\begin{multline*}
H^1(G_{M_{\infty}},\Fpbar)=H^1(G_{l((u))},\Fpbar) 
= \Hom(G_{l((u))}, \Fbar_p) \simeq \bigl(l((u))/\psi l((u))\bigr)
\otimes_{\Fp} \Fpbar;
\end{multline*}
concretely, 
the element~$a\in l((u))$ corresponds to the
homomorphism~$f_a:G_{l((u))}\to\Fp$ given by
$f_a(g)=g(x)-x$, where $x \in l((u))^{\sep}$ is chosen
so that $\psi(x)=a$. (See e.g.\ ~\cite[X \S3(a)]{MR554237} for more details.)
\begin{thm}\label{thm: schmid reciprocity}
 Let ~$\sigma_b \in \Gal(M^{(p)}_{\infty}/M_{\infty})$ be the Galois element
 corresponding via the local Artin map
 to an element~$b \in l((u))^{\times}\otimes \Fp$, and let
  $f_a$ be the element of~$H^1(G_{M_{\infty}},\Fpbar)$ corresponding
  to an element
  ~$a \in \bigl(l((u))/\psi l((u))\bigr)\otimes_{\Fp} \Fpbar$. 
  Then 
  $$\langle f_a, \sigma_b \rangle = \Tr_{l\otimes_{\Fp} \Fpbar/\Fpbar}\left(\Res a \cdot \frac{db}{b}\right).$$
\end{thm}
\begin{proof}This was first proved in~\cite{MR1581502}; for a more
  modern proof, see~\cite[XIV Cor.\ to Prop.\ 15]{MR554237}.
  \end{proof}

\subsection{Crystalline extension classes and $\LBDJ$}\label{subsec:
  etale phi modules from GLS}
We begin by briefly recalling some of the main results of~\cite{MR3164985}. For
each~$0\le i\le f-1$ we fix an integer~$r_i\in[1,p]$; we then
define~$r_i$ for all integers~$i$ by demanding that~$r_{i+f}=r_i$. We
let~$J$ be a subset of~$\{0,\dots,f-1\}$, and we assume that~$J$ is
maximal in the sense of~\cite[\S7.2]{2016arXiv160307708D}; in other
words, we assume that:
\begin{itemize}
\item if for some $i>j$ we have $(r_j,\dots,r_i)=(1,p-1,\dots,p-1,p)$,
  and $j+1,\dots,i\notin J$, then~$j\notin J$; and
\item if all the $r_i$ are equal to~$p-1$, or if~$p=2$ and all of
  the~$r_i$ are equal to~$2$, then~$J$ is nonempty.
\end{itemize}
We let~$\chi:G_K\to\Fpbartimes$ be a character with the property
that\[\chi|_{I_K}=\prod_{j\in J}\omega_j^{r_j}\prod_{j\notin
    J}\omega_j^{-r_j}.\] We let~$\LBDJ$ denote the subset
of~$H^1(G_K,\chi)$ consisting of those classes  corresponding to
extensions of the trivial character by~$\chi$ that arise as the
reductions of crystalline representation whose $\sigma_i$-labelled
Hodge--Tate weights are $\{0,(-1)^{i\notin J}r_i\}$, where
$(-1)^{i\notin J}$ is $1$ if $i\in J$ and $-1$ otherwise. It follows
from the proof of~\cite[Thm.\ 9.1]{MR3164985}, together
with~\cite[Lem.\ 9.3, Lem.\ 9.4]{MR3164985} and (in the case that~$p=2$) the results
of~\cite{Wangp2} that:
\begin{itemize}
\item $\LBDJ$ is an~$\Fpbar$-subspace of~$H^1(G_K,\chi)$.
\item An extension class is in~$\LBDJ$ if and only if it admits a
  \emph{reducible} crystalline lift whose $\sigma_i$-labelled
Hodge--Tate weights are $\{0,(-1)^{i\notin J}r_i\}$.
\item If $J=\{0,\dots,f-1\}$ and all~$r_i=p$, then~$\LBDJ=H^1(G_K,\chi)$.
\item Assume that we are not in the case of the previous bullet point.
  Then $\dim_{\Fpbar}\LBDJ=|J|$, unless~$\chi=1$, in which case $\dim_{\Fpbar}\LBDJ=|J|+1$.
\end{itemize}

We recall below from~\cite{2016arXiv160307708D} the definition of
another subspace of~$H^1(G_K,\chi)$, denoted $\LDDR$; our main result, then, is
that~$\LBDJ=\LDDR$. We begin with an easy special case.

\begin{lem}
  \label{lem: cyclotomic J=S case of main thm}If
  $J= \{0,\dots,f-1\}$ and every~$r_i=p$ then $\LBDJ=\LDDR$.
\end{lem}
\begin{proof}
  In this case we have ~$\LDDR=H^1(G_K,\chi)$ by definition (see
  Definition~\ref{defn: LDDR} below), and we already noted above that
  $\LBDJ=H^1(G_K,\chi)$.
\end{proof}

We  can and do exclude the case covered by~Lemma~\ref{lem: cyclotomic J=S case of main thm}
from now on; that is, in addition to the assumptions made above, we
assume that:
\begin{itemize}
\item if every~$r_i$ is equal to~$p$, then~$J\ne\{0,\dots,f-1\}$.
\end{itemize}

If~$\chi=\epsilonbar$ then the peu ramifi\'e subspace
of~$H^1(G_K,\epsilonbar)$ is by definition the codimension one subspace spanned by
the classes corresponding via Kummer theory to elements
of~$\cO_K^\times$. Since we have excluded the cases covered by
Lemma~\ref{lem: cyclotomic J=S case of main thm}, $\LBDJ$ is contained in the peu ramifi\'e subspace
  of~$H^1(G_K,\epsilonbar)$ by~\cite[Thm.\ 4.9]{MR3352531}. 

By~\cite[Lem.\ 5.4.2]{MR3324938}, for any~$\chi\ne\epsilonbar$ the natural
restriction map $H^1(G_K,\chi)\to H^1(G_{K_\infty},\chi)$ is
injective, while if~$\chi=\epsilonbar$ then the kernel is spanned
by the tres ramifi\'e class corresponding to~$-p$; in particular, the
restriction of this map to~$\LBDJ$ is injective. The following theorem
describes the image of~$\LBDJ$; before stating it, we introduce some
notation that we will use throughout the paper.

Write~$\chi$ as a power
of~$\omega_0$ 
times an unramified
character~$\mu:\Gal(L/K)\to\Fpbartimes$, and write~$\mu(\Frob_K)=a$,
so that~$a^{[l:k]}=1$; here $\Frob_K\in\Gal(L/K)$ denotes the arithmetic
Frobenius. For
each~$\sigma:k\into\Fpbar$, we let~$\lambda_{\sigma,\mu}$ be the
element $(1,a^{-1},\dots,a^{1-[l:k]})\in l\otimes_{k,\sigma}\Fpbar$,
so that~$\lambda_{\sigma,\mu}$ is a basis of the one-dimensional $\Fpbar$-vector space
$(l\otimes_{k,\sigma}\Fpbar)^{\Gal(L/K)=\mu}$. Similarly, we let $\lambda_{\sigma,\mu^{-1}}$ be the
element $(1,a,\dots,a^{[l:k]-1})\in l\otimes_{k,\sigma}\Fpbar$.

\begin{thm}\label{thm: form of Kisin modules}The subspace~$\LBDJ$ of 
	$H^1(G_K,\chi)$ consists of precisely those classes whose restrictions
  to $H^1(G_{K_\infty},\chi)$ can be represented by \'etale
  $\varphi$-modules~$\cM$ of the following form:

 Set $h_i = r_i$ if $i \in
  J$ and $h_i = 0$ if $i \not\in J$.  Then 
  we can choose bases $e_i,f_i$
  of the $\cM_i$ so that $\varphi$  has the form
  \begin{eqnarray*}
  \varphi(e_{i-1}) & =&  u^{r_i-h_i} e_{i}\\
    \varphi(f_{i-1})& =& (a)_i u^{h_i} f_{i} + x_i e_{i}
 \end{eqnarray*}
 Here
 $(a)_i=1$ for $i\ne 0$, and equals $a=\mu(\Frob_K)$
for~$i=0$; and we have $x_i = 0$ if $i \not\in J$ and $x_i \in \Fpbar$
if $i \in
J$,
except in the case that~$\chi=1$. 

If $\chi=1$ then~$a=1$, and if we fix some $i_0 \in J$, then $x_{i_0}$ is
allowed to be of the form $x_{i_0}'+x_{i_0}''u^p$ with
$x'_{i_0},x''_{i_0}\in\Fpbar$ \emph{(}while the other~$x_i$ are in~$\Fpbar$\emph{)}.

In every case, the~$x_i$ are uniquely determined by~$\cM$. 
\end{thm}
\begin{proof}
  In the case~$p>2$, this is an immediate consequence of ~\cite[Thm.\
  7.9]{MR3164985} (which describes the corresponding Kisin modules,
  which are just lattices in~$\cM$; the set~$J'$ appearing there can
  be taken to be our~$J$ by~\cite[Prop.\
  8.8]{MR3164985} and our assumption that~$J$ is maximal) and the proof of~\cite[Thm.\
  9.1]{MR3164985} (which shows that the different~$x_i$ give rise to
  different Galois representations), while if~$p=2$, then the result
  follows from the results of~\cite{Wangp2}.
  \end{proof}As in Section~\ref{sec:notation}, we let~$\pi$ be a choice
of $(p^f-1)$st root of~$-p$.
Write~$M:=L(\pi)$, where~$L/K$ is an unramified extension of
degree prime to~$p$, chosen so that~$\chi|_{G_M}$ is
trivial. (In~\cite{2016arXiv160307708D} a slightly
more general choice of~$M$ is permitted, but it is shown there that
their constructions are independent of this choice, and this
choice is convenient for us.) Then~$M/K$ is an abelian extension of degree prime
to~$p$. Since~$(p^f-1)$ is prime to~$p$, for each~$n\ge 1$ there is a
unique $p^n$th root~$\pi^{1/p^n}$ of~$\pi$ such 
that~$(\pi^{1/p^n})^{(p^f -1)}=(-p)^{1/p^n}$, and we set~$M_n=M(\pi^{1/p^n})$, $M_\infty=\cup_nM_n$.

 If~$\cM$ is an \'etale $\varphi$-module with corresponding
$G_{K_\infty}$-representation~$T(\cM)$, then it is easy to check that the \'etale
$\varphi$-module corresponding to~$T(\cM)|_{G_{M_\infty}}$
is \[\cM_M:=l((u))\otimes_{k((u)),u\mapsto u^{p^{f}-1}}\cM.\]  

Applying this to one of the \'etale $\varphi$-modules arising in
the statement of Theorem~\ref{thm: form of Kisin modules},
it follows that (with the obvious choice of basis $e_i,f_i$ for $\cM_M$) the matrix of
$\varphi:\cM_{M,i-1}\to \cM_{M,i}$ is
\[
  \begin{pmatrix}
    u^{(r_i-h_i)(p^f-1)} & x_i\\ 0& (a)_iu^{h_i(p^f-1)}
  \end{pmatrix}
\]
where as above $h_i=r_i$ if $i\in J$ and $h_i=0$ if
$i\notin J$, and $x_i$ is zero if $i\notin
J$. 
Furthermore~$x_i\in\Fpbar$, except that if~$\chi=1$, we have fixed
a choice of~$i_0\in J$, and~$x_{i_0}$ is
allowed to be of the form $x_{i_0}'+x_{i_0}''u^{p(p^f-1)}$ with
$x'_{i_0},x''_{i_0}\in\Fpbar$. (Here the $\cM_{M,i}$
are periodic with period~$f[l:k]$, but of course the $r_i, h_i$ and $x_i$ depend
only on ~$i$ modulo~$f$.) 

We now make a change of basis, setting
$e_i'=u^{\alpha_i}e_i$ and $f'_i=a^{\lfloor i/f\rfloor
}u^{\beta_i}f_i$ (where $0\le i\le f[l:k]-1$), so that the matrix of
$\varphi:\cM_{M,i-1}\to \cM_{M,i}$ becomes \[
  \begin{pmatrix}
    u^{(r_i-h_i)(p^f-1)+p\alpha_{i-1}-\alpha_i} & a^{\lfloor i-1/f\rfloor}x_iu^{p\beta_{i-1}-\alpha_i}\\ 0& u^{h_i(p^f-1)+p\beta_{i-1}-\beta_i}
  \end{pmatrix}\]We  choose the $\alpha_i,\beta_i$ so that the
entries on the diagonal become trivial; concretely, this means that we
set \[\alpha_i=-\sum_{j=0}^{f-1}(r_{i+j+1}-h_{i+j+1})p^{f-1-j}\, ,
	\quad \beta_i=-\sum_{j=0}^{f-1}h_{i+j+1}p^{f-1-j} \, .\]
Write~$\xi_i:=\alpha_i-p\beta_{i-1}$,
so that we have
\[\xi_i = \sum_{j=0}^{f-1}(-1)^{i+j+1\notin J}r_{i+j+1}p^{f-1-j}+\delta_{i\in J}r_i(p^f-1)\]
where 
$\delta_{i\in J}=1$ if $i\in J$ and $0$ otherwise.

With the obvious
basis for~$\cM_M$ as an
$l((u))\otimes_{\Fp} \Fpbar$-module, 
~$\phi_{\cM_M}$ is given by the
matrix \[
  \begin{pmatrix}
    1& (x_ia^{-1}\lambda_{\sigma_i,\mu^{-1}} u^{-\xi_i})_{i = 0,\dots,f-1}\\0&1
  \end{pmatrix}
\]where~$\lambda_{\sigma_i,\mu^{-1}}$ is the element
of~$l\otimes_{k,\sigma_i}\Fpbar$ that we defined above. 
Then~$T(\cM_M)$ is an extension of the
trivial representation by itself, and thus corresponds to an
element of~$\Hom(G_{l((u))},\Fpbar)$. By the  definition of~$T$, the
kernel of this homomorphism corresponds to the Artin--Schreier
extension of~$l((u))$ determined by~$(x_i\lambda_{\sigma_i,\mu^{-1}} u^{-\xi_i})_{i = 0,\dots,f-1}$. We have
therefore proved the following result.

\begin{cor}
  \label{cor: description of LBDJ classes as artin schreier
    extensions} 
    The image of~$\LBDJ$
  in~$H^1(G_{M_\infty},\Fpbar)=\Hom(G_{l((u))},\Fpbar)$ is spanned 
  by the classes~$f_{\lambda_{\sigma_i,\mu^{-1}} u^{-\xi_i}}$ corresponding via Artin--Schreier theory to the
  elements
 $$\lambda_{\sigma_i,\mu^{-1}} u^{-\xi_i} \in l\otimes_{k,\sigma_i} \Fbar_p 
  \subseteq l\otimes_{\F_p} \Fbar_p,$$
  for~$i\in J$, together with the
  class~$f_{\lambda_{\sigma_{i_0},\mu^{-1}}u^{p(p^f-1)-\xi_{i_0}}}$ if~$\chi=1$.
  \end{cor}

As in~\cite[\S 3.2]{2016arXiv160307708D}, we
may write $\chi|_{I_K}=\omega_0^{n_0}$ for some unique~$n_0$ of the
form~$n_0=\sum_{j=1}^fa_jp^{f-j}$ with each~$a_j\in[1,p]$ and at least
one~$a_j\ne p$. We set \[n_i=\sum_{j=1}^fa_{i+j}p^{f-j}, \]  so we
have~$\chi|_{I_K}=\omega_i^{n_i}$, and for all $i,j$ we have \[p^{-i}n_i\equiv p^{-j}n_j\pmod{p^f-1}.\]

Note that we
have \begin{align*}\chi|_{I_K}&=\prod_{j\in
                                J}\omega_j^{r_j}\prod_{j\notin J}\omega_j^{-r_j}
                               \\ 
                                                                  &=\prod_{
                                            j=0}^{f-1}\omega_{i}^{-(-1)^{i+j+1\in
                                            J}r_{i+j+1}p^{f-1-j}} \\
                                             &=\omega_{i}^{\alpha_i-p\beta_{i-1}}
		      = \omega_i^{\xi_i},\end{align*}
		      so that in particular we have
\numequation
\label{eqn:congruence}
\xi_i\equiv n_i\pmod{p^f-1}.
\end{equation}

\subsection{The Artin--Hasse exponential and $\LDDR$}\label{subsec: Artin Hasse}
We now recall some of the definitions made
in~\cite[\S5.1]{2016arXiv160307708D}. In particular, for each~$i$ we
define an embedding~$\sigma_i'$ 
and an integer $n'_i$ as follows. If $a_{i-1}\ne p$, then we set
$\sigma'_i=\sigma_{i-1}$ and $n'_i=n_{i-1}$. If $a_{i-1}=p$,
then we let $j$ be the greatest integer less than~$i$ such
that~$a_{j-1}\ne p-1$, and we set $\sigma'_i=\sigma_{j-1}$
and~$n'_i=n_{j-1}-(p^f-1)$. Note that we always have $n'_i>0$.

We let $E(x)=\exp(\sum_{m\ge 0}x^{p^m}/p^m)\in\Zp[[x]]$ denote the
Artin--Hasse exponential. For any~$\alpha\in\m_M$, we define the
homomorphism
\[\epsilon_\alpha:l\otimes_{\Fp} \Fbar_p\to\cO_M^\times\otimes_{\Fp}\Fpbar\]
by $\epsilon_\alpha(a\otimes b):=E([a]\alpha)\otimes b$, 
where $[\cdot]:l\to W(l)$ is the Teichm\"uller lift. 
Then we
set \[u_i:=\epsilon_{\pi^{n'_i}}(\lambda_{\sigma'_i,\mu}) \in
  \cO_{M}^\times\otimes\Fpbar.\] In the case that~$\chi=1$, we also
set $u_{\triv}:=\pi\otimes 1\in M^\times\otimes\Fpbar$, and in the
case that~$\chi=\epsilonbar$, the mod~$p$ cyclotomic character, 
we set $u_{\cyc}:=\epsilon_{\pi^{p(p^f-1)/(p-1)}}(b\otimes 1)$, where
$b\in l$ is any element with $\operatorname{Tr}_{l/\Fp}(b)\ne 0$. It is shown
in~\cite[\S5]{2016arXiv160307708D} that the~$u_i$, together
with~$u_{\triv}$ if~$\chi=1$, and~$u_{\cyc}$ if~$\chi=\epsilonbar$,
are a basis of the $\Fpbar$-vector
space \[U_\chi:=\bigl(M^\times\otimes\Fpbar(\chi^{-1})\bigr)^{\Gal(M/K)}.\]

Via the Artin map~$\Art_M$, we may write \[H^1(G_K,\chi)\cong
  \Hom_{\Gal(M/K)}\bigl(M^\times,\Fpbar(\chi)\bigr)\]and thus identify
$H^1(G_K,\chi)$ with the $\Fpbar$-dual of~$U_\chi$. We then define a
basis of $H^1(G_K,\chi)$ by letting ~$c_i$, $c_{\triv}$ (if~$\chi=1$) and~$c_{\cyc}$
(if~$\chi=\epsilonbar$) denote the dual basis to that given by
the~$u_i$, $u_{\triv}$, $u_{\cyc}$.

Recall from~\cite[\S7.1]{2016arXiv160307708D} the definition
of 
the set~$\mu(J)$. It is defined as follows: $\mu(J)=J$, unless there
is some $i\notin J$ for which we have 
$a_{i-1}=p$,
$a_{i-2}=p-1,\dots,a_{i-s}=p-1,a_{i-s-1}\ne p-1$, and at least one of
$i-1,i-2,\dots,i-s$ is in~$J$. If this is the case, we let $x$ be
minimal such that $i-x\in J$, and we consider the set obtained
from~$J$ by replacing~$i-x$ with~$i$. Then~$\mu(J)$ is the set
obtained by simultaneously making all such replacements (that is,
making these replacements for all possible~$i$).

\begin{df}\label{defn: LDDR}
We define~$\LDDR$ to be the subspace of~$H^1(G_K,\chi)$ spanned
by the classes $c_i$ for~$i\in \mu(J)$, together with the
class~$c_{\triv}$ if~$\chi=1$, and the class~$c_{\cyc}$
if~$\chi=\epsilonbar$, $J=\{0,\dots,f-1\}$ and every~$r_i=p$.
\end{df}

\subsection{The comparison of $\LBDJ$ and $\LDDR$}\label{subsec: explicit
  reciprocity}In this section we will prove that the classes in~$\LBDJ$
are orthogonal to certain of the~$u_i$. 
We begin with a computation
that will allow us to
compare the constructions underlying the definition of $\LDDR$, which involve the
Artin--Hasse exponential,
with the field of norms constructions underlying the description of $\LBDJ$.

\begin{lem}
  \label{lem: norm compatible AH} For any~$n\ge 1$, $a\in l$, and $r\ge
  1$ with~$(r,p)=1$ we have $N_{K_n/K}E\bigl([a^{1/p^n}](\pi^{1/p^n})^r\bigr) = E([a]\pi^r)$.
\end{lem}
\begin{proof}
Let~$\zeta$ be a primitive~$p^n$th root of unity. Then
$$
\begin{aligned}
N_{K_n/K}E\bigl([a^{1/p^n}](\pi^{1/p^n})^r\bigr)
= & \ \prod_{k=0}^{p^n - 1}  E\bigl([a^{1/p^n}](\pi^{1/p^n})^r \zeta^k\bigr) \\
= & \ \prod_{k=0}^{p^n - 1}  \exp \left(\sum_{m \ge 0}
\frac{ [a^{1/p^n}]^{p^m} (\pi^{1/p^n})^{r p^m} \zeta^{k p^m}}{p^m} \right) \\
= & \ \exp\left(  \sum_{k=0}^{p^n - 1}  \sum_{m \ge 0}
\frac{ [a^{1/p^n}]^{p^m} (\pi^{1/p^n})^{r p^m} \zeta^{k p^m}}{p^m} \right) \\
= & \ \exp\left(   \sum_{m \ge 0} 
\frac{ [a^{1/p^n}]^{p^m} (\pi^{1/p^n})^{r p^m}}{p^m}  \sum_{k=0}^{p^n - 1}  \zeta^{k p^m}  \right).
 \end{aligned}
$$
Now the sum over roots of unity is~$0$ if~$\zeta^{p^m} \ne 1$ (equivalently, $m < n$)
and~$p^n$ if~$\zeta^{p^m} = 1$ (equivalently, $m \ge n$). Hence
\begin{align*}
N_{K_n/K}E\bigl([a^{1/p^n}](\pi^{1/p^n})^r\bigr)
= & \ \exp\left(  \sum_{m \ge n}   
\frac{ [a^{1/p^n}]^{p^m} (\pi^{1/p^n})^{r p^m} p^n}{p^m} \right) \\
= & \ \exp\left(  \sum_{m \ge 0}   
\frac{ [a^{1/p^n}]^{p^{n+m}} (\pi^{1/p^n})^{r p^{n+m}} p^n}{p^{m+n}} \right) \\
= & \ \exp\left(  \sum_{m \ge 0}  
\frac{ [a]^{p^{m}} (\pi^{r})^{p^{m}}}{p^{m}} \right) =  E([a] \pi^r). \qedhere
 \end{align*}\end{proof}

For each~$r\ge 1$ have a
homomorphism \[\epsilon_{u^r}:l\otimes\Fpbar\to
  l((u))^\times\otimes\Fp\] defined by $\epsilon_{u^r}(a\otimes
b)=E(au^r)\otimes b$. Then for each~$i$ we
set \[\tilde{u}_i:=\epsilon_{u^{n'_i}}(\lambda_{\sigma_i',\mu})\in l((u))^\times\otimes\Fp.\]

\begin{lemma} \label{lemma: tilting AH}
Let~$r\ge 1$ be coprime to~$p$. Then under the homomorphism~{\em (\ref{eqn:map on norms})} {\em (}with $M$ in place of $K${\em )},
the image of~$E([a]u^r)$ is equal to~$E([a] \pi^r)$; consequently, for
each~$i$, the
image of~$\tilde{u}_i$ is~$u_i$.
\end{lemma}

\begin{proof} This is an immediate consequence of
 Lemma~\ref{lem: norm compatible AH},
 taking into account Lemma~\ref{lem: n'_i is a unit} below, which shows 
 that~$n'_i$ is
 coprime to~$p$. 
\end{proof}

We now state and prove our main result,
which establishes~\cite[Conj.\ 7.2]{2016arXiv160307708D}, 
by reducing  the equality~$\LDDR=\LBDJ$ to a purely combinatorial
problem that is solved in Section~\ref{subsec: combinatorics}.

\begin{thm}\label{BDJ equals DDR}
  $\LBDJ=\LDDR$.
  \end{thm}
\begin{proof}
  Since we
  have~$\dim_{\Fpbar}\LBDJ=\dim_{\Fpbar}\LDDR=|J|+\delta_{\chi=1}$,
  it is enough to prove that~$\LBDJ\subseteq\LDDR$. By the definition
  of~$\LDDR$, it is equivalent to prove 
  that the image of every class in~$\LBDJ$ in~$H^1(G_M,\Fpbar)$ is
  orthogonal under the pairing of Section~\ref{subsec: compatibility of pairings} to the
  elements~$u_j\in U_\chi$, $j\notin \mu(J)$.

(In the case
  that~$\chi=\epsilonbar$, we also need to show that the classes
  are orthogonal to~$u_{\cyc}$; to see this, note that, as explained
  in~\cite[\S6.4]{2016arXiv160307708D} the classes~$c_i$ (together
  with~$c_{\triv}$ if~$p=2$) span the space of classes which are
  (equivalently) flatly or typically ramified in the sense
  of~\cite[\S3.3]{2016arXiv160307708D}, which are exactly the peu
  ramifi\'e classes; in other words, the classes orthogonal
  to~$u_{\cyc}$ are exactly the peu ramifi\'e classes. As we recalled
  in Section~\ref{subsec:
  etale phi modules from GLS}, it follows
  from~\cite[Thm.\ 4.9]{MR3352531} that every class in~$\LBDJ$ is peu ramifi\'e.)

  Combining Lemma~\ref{lem: class field theory field of norms},
  Lemma~\ref{lem: pairing compatibility}, Theorem~\ref{thm: schmid
    reciprocity}, Lemma~\ref{lemma: tilting AH}, and Corollary~\ref{cor: description of LBDJ classes
    as artin schreier extensions}, we see that we must show that for
  all $i\in J$, $j\notin \mu(J)$, the residue 
  \numequation\label{eqn: residue pairing to compute} \Tr_{l\otimes_{\Fp}\Fpbar/\Fpbar}\Res \left( \dlog(\tilde{u}_j)
    \cdot \lambda_{\sigma_i,\mu^{-1}}  u^{-\xi_i}
  \right).\end{equation}vanishes. (If~$\chi=1$ then we must also show
that the pairing with~$\lambda_{\sigma_{i_0},\mu^{-1}}  u^{p(p^f-1)-\xi_{i_0}}$ vanishes.)

  Since
  $$\dlog E(X) = \left(X + X^p + X^{p^2} + \ldots \right) \dlog X$$
  and~$\dlog(\lambda u^{n}) = n \cdot u^{-1}$, 
  the pairing~(\ref{eqn: residue pairing to compute}) evaluates to
  $$ \Tr_{l\otimes_{\Fp}\Fpbar/\Fpbar}\Res \left( \sum_{m \ge 0} n'_j
    (\varphi\otimes 1)^m(\lambda_{\sigma'_j,\mu}) u^{n'_j p^m - 1}
    \cdot \lambda_{\sigma_i,\mu^{-1}} u^{-\xi_i}
  \right).$$(Here~$\varphi\otimes 1:l\otimes\Fpbar\to l\otimes\Fpbar$
  is the $p$-th power map on~$l$.) 

This residue is given by the coefficient of~$u^{-1}$, 
so we see that this pairing can be non-zero 
only when~$\xi_i = p^m n'_j$ for some~$m\ge 0$. (If~$\chi=1$ then we
must also consider the possibility that $\xi_i-p(p^f-1) =
p^m n'_j$, but this is excluded by Lemma~\ref{lem: extra
  pairing lemma trivial} below.) If this holds, then the pairing
evaluates to \[n'_j\Tr_{l\otimes_{\Fp}\Fpbar/\Fpbar}
    (\varphi\otimes 1)^m(\lambda_{\sigma'_j,\mu}) 
    \cdot \lambda_{\sigma_i,\mu^{-1}}.\]  Now, we have \[
  (\varphi\otimes 1)^m(\lambda_{\sigma'_j,\mu})\cdot
  \lambda_{\sigma_i,\mu^{-1}} = (\varphi\otimes
  1)^m(\lambda_{\sigma'_j,\mu}\lambda_{\sigma_{i-m},\mu^{-1}})\] which
is nonzero 
if and only if ~$\sigma'_j = \sigma_{i - m}$, in which case its trace
to~$\Fpbar$ is equal to~$[l:k]$.

In conclusion, we have seen that in order
for the pairing to be non-zero,
we require 
  \begin{itemize}
  \item $\sigma'_j=\sigma_{i-m}$, and
    \item $\xi_i=p^m n'_j$.
  \end{itemize} (In fact, although we don't need this stronger
  statement,
  we observe that the pairing is non-zero if and only if
  these conditions hold,
  because~$n'_j$ is always a unit by
  Lemma~\ref{lem: n'_i is a unit}, while~$[l:k]$ is prime to~$p$.) 
  By Proposition~\ref{prop: when can pairing be nonzero} below, these
  conditions 
  imply that~$j\in\mu(J)$, as required.
\end{proof}
\begin{rem}
  \label{rem: comparison of bases}It is clear that the method of proof of
  Theorem~\ref{BDJ equals DDR} could be used to compare the bases 
  of~$\LBDJ$ and~$\LDDR$ that we have been working with.  
  We have checked that
  in suitably generic cases the bases are the same
  (up to scalars), but that in exceptional cases they may differ.
  \end{rem}
\subsection{Combinatorics}\label{subsec: combinatorics}

Our main aim in this section is to prove Proposition~\ref{prop: when
  can pairing be nonzero}, which was used in the proof of
Theorem~\ref{BDJ equals DDR}. We begin with some simple observations;
the following three lemmas give us some control on the
quantities~$\xi_i$ and~$n'_i$ which will be important in the proof
of Proposition~\ref{prop: when can pairing be nonzero}. 
\begin{lem}\label{lem: n'_i is a unit}
$n'_i$ is not divisible by~$p$.
\end{lem}
\begin{proof}
  This is automatic if~$a_{i-1} \ne p$ because
  then~$n'_i =n_{i-1}
   \equiv a_{i-1} \pmod p$.  Assume
  that~$a_{i-1} = p$, and write
  that~$(a_{i-1},a_{i-2},\ldots,a_{j}) = (p,p-1,\ldots,p-1)$, with~$a_{j-1} \ne p-1$.  Now
$$n'_{i} := n_{j-1}-(p^f-1)\equiv n_{j - 1} + 1\equiv a_{j - 1} + 1 \pmod p.$$
However, since~$a_{j-1} \ne p-1$ and lies in~$[1,p]$, we
have~$a_{j-1} \not\equiv -1 \mod p$, and
so~$n'_i \not\equiv 0 \pmod p $.
\end{proof}

\begin{lem}\label{lem: bounds on xi}If $i\in J$ then $0<\xi_i< p^2(p^f-1)/(p-1)$.
\end{lem}
\begin{proof}
  Since $i \in J,$ we have
  \numequation
  \label{eqn:xi formula}
  \xi_i=p^f r_i +(-1)^{i+1\notin J}p^{f-1}r_{i+1}+(-1)^{i+2\notin
      J}p^{f-2}r_{i+2}+\dots+(-1)^{i-1\notin J}pr_{i-1}.
\end{equation}
  The upper
  bound is immediate, as we have $r_j\le p$ for all~$j$ (and in the
  case that all~$r_j$ are equal to~$p$, we are not allowing~$J^c$ to
  be empty). For the lower
  bound, if $r_i\ge 2$ then $\xi_i\ge 2p^f-(p^f+p^{f-1}+\dots+p^2)>0$,
  so we may assume that $r_i=1$. Suppose that $J\ne\{i\}$, and let
  $x\ge 0$
  be minimal so that $i+x+1\in J$. Since $r_i=1$ and~$i\in J$, it follows from the
  maximality condition on~$J$
  that no initial segment of~$(r_{i+1}, \ldots, r_{i+x})$
  can be~$(p-1,p-1,\ldots,p)$  (which also
  excludes the degenerate case consisting of a single initial~$p$). 
  Hence either all the~$r_{j}$ for~$j \in [i+1,i+x]$ are at most~$p-1$, in which case
  \[p^{f-1}r_{i+1}+\dots+p^{f-x}r_{i+x}\le
    (p^{f-1}+\dots+p^{f-x})(p-1)=p^f-p^{f-x},\] so that
  \[\xi_i\ge
  p^{f-x}+p^{f-x-1}-(p^{f-x-2}+\dots+p)p=p^{f-x}-p^{f-x-2}-\dots-p^2>0,\]
  or for some~$y < x$ we have~$r_{i+1}, \ldots, r_{i + y} = p-1$ and~$r_{i+y+1} < p-1$,
  in which case
$$\begin{aligned} p^{f-1}r_{i+1}+\dots+p^{f-x}r_{i+x}\le & \ 
    (p^{f-1}+\dots+p^{f-y})(p-1) \\ & \quad + (p-2) p^{f-y-1} + 
    p(p^{f-y-2} + \ldots p^{f-x}) \\
    = & \ (p^{f-1}+\dots+p^{f-x})(p-1) \\ & \quad - p^{f-y-1} + p^{f-y-2} + \ldots + p^{f-x} \\
    \le & \ (p^{f-1}+\dots+p^{f-x})(p-1) \\
   = & \ p^f-p^{f-x}, \end{aligned}$$  
     and one proceeds as above.
Finally, if~$J=\{i\}$, then arguing as above (and again using the
maximality condition on~$J$) we see (considering the two cases as above)
 that $\xi_i\ge p^f-(p^{f-1}+\dots+p)(p-1)=p>0$.
\end{proof}
\begin{lem}\label{lem: bounds on n_i}
	For any value of $i$, we have
  $(p^f-1)/(p-1)\le n_i< (p^f-1)+(p^f-1)/(p-1)$.
\end{lem}
\begin{proof}
  This is immediate from the definition of~$n_i$.
\end{proof}
Let~$v_p(\xi_i)$ denote the $p$-adic valuation of~$\xi_i$.
The following lemma shows that~$\xi_i$
is in some sense a function of this valuation, and is crucial for our main argument.             

\begin{lem}\label{lem: xi_i in terms of n_l}
 If~$i\in J$, and if~$m:= v_p(\xi_i)$,
 then~$m \ge 1$.
   If furthermore $m>1$, then we have
 $\xi_i=p^m\bigl(n_{i-m}-(p^f-1)\bigr)$,
  while if $m=1$, then either $\xi_i=pn_{i-1}$ or
  $\xi_i=p\bigl(n_{i-1}-(p^f-1)\bigr)$, depending on whether or not
  $\xi_i/p\ge(p^f-1)/(p-1)$.
\end{lem}
\begin{proof}
Equation~(\ref{eqn:xi formula}) shows  that~$m$ is at least~$1$ if $i \in J$. 
From (\ref{eqn:congruence}), we deduce
 that $\xi_i/p^m\equiv
 n_{i-m}\pmod{p^f-1}$. By Lemma~\ref{lem: bounds on xi} we
 have \[0<\xi_i/p^m< p^{2-m}(p^f-1)/(p-1),\] so that if $m\ge 2$ it
 follows by  Lemma~\ref{lem: bounds on n_i} that \[\xi_i/p^m <
   (p^f-1)/(p-1)\le n_{i-m}< (p^f-1)+(p^f-1)/(p-1).\]
   Since~$\xi_i > 0$ by Lemma~\ref{lem: bounds on xi}, the congruence
   modulo~$p^f - 1$ forces the equality
 $n_{i-m}-\xi_i/p^m=(p^f-1)$.
  If  $m=1$, then we have \[0<\xi_i/p <
   (p^f-1)+(p^f-1)/(p-1)\]and the claim follows in the same way.
\end{proof}
The following simple lemma was used in the proof of Theorem~\ref{BDJ
  equals DDR} in the case~$\chi=1$.
\begin{lem}
  \label{lem: extra pairing lemma trivial}Suppose that~$\chi=1$ and that
  $i\in J$. Then there are no solutions to the equation
  \begin{itemize}
  \item $\xi_i-p(p^f-1)=p^m (p^f-1)$.
  \end{itemize}for any~$m\ge 0$.
\end{lem}
\begin{proof} 
  Since~$\chi=1$, we have $n_j=p^f-1$ for all~$j$.
  From
  Lemma~\ref{lem: xi_i in terms of n_l}, we find that
  either~$v_p(\xi_i) \ge 2$, in which case~$\xi_i = 0$
  (contradicting Lemma~\ref{lem: bounds on xi}), or~$v_p(\xi_i) = 1$,
  in which case either~$\xi_i = 0$ or~$\xi_i = p(p^f - 1)$.
  The first case again contradicts Lemma~\ref{lem: bounds on xi}. The 
  second case leads to the equation  $0=p^m(p^f-1)$, which has
  no solutions, as required.
  \end{proof}

We now prove our main combinatorial result.
\begin{prop}
  \label{prop: when can pairing be nonzero}Suppose that~$i\in J$, and
  that for some integers $j,m$ we have
  \begin{itemize}
  \item $\sigma'_j=\sigma_{i-m}$, and
   \item $\xi_i=p^m n'_j$.
  \end{itemize}

   Then~$j\in\mu(J)$.
\end{prop}
\begin{proof} By Lemma~\ref{lem: n'_i is a unit}, we must have
  $m=v_p(\xi_i)$. Suppose firstly that $m=1$ and $\xi_i=pn_{i-1}$. We
  need to solve the equations $\sigma'_j=\sigma_{i-1}$ and
  $n'_j=n_{i-1}$.

  If $a_{j-1}=p$, then we have $\sigma'_j=\sigma_{s-1}$
  and $n'_j=n_{s-1}-(p^f-1)$, where~$s$ is
  the greatest integer less than~$j$ for which $a_{s-1}\ne p-1$.
  Since $\sigma'_j = \sigma_{i-1}$ by assumption, we find that $s = i$.
  But then 
  $n_{i-1}=n'_j=n_{i-1}-(p^f-1)$,
  which is not possible.

  Thus $a_{j-1}\ne p$, and hence 
  we have $\sigma'_j=\sigma_{j-1}$, so that~$j=i$. 
  We must show 
  that~$j = i \in\mu(J)$. By the
definition of~$\mu(J)$, this will be the case unless for some $s>i$ we
have $i+1,\dots,s\notin J$, and
$(a_{i},\dots,a_{s-1})=(p-1,\dots,p-1,p)$. Suppose then that
this holds; we must show that we cannot have $\xi_i=pn_{i-1}$ after
all. 
Now, by definition and the assumption that $i+1,\dots,s\notin J$ we
have \begin{align*}\xi_i/p&=p^{f-1} r_i-p^{f-2}r_{i+1}-\dots
  +(-1)^{s+1\notin J}p^{f+i-2-s}r_{s+1}+\dots+(-1)^{i-1\notin
    J}r_{i-1}\\ &\le p^f -(p^{f-2}+\dots+
                  p^{f+i-s-1})+(p^{f+i-2-s}+\dots+1)p\\&=
                                                       p^f-(p^{f-2}+\dots
                                                       + p^{f+i-s})+(p^{f+i-2-s}+\dots+p)\end{align*}
                while \begin{align*}n_{i-1}&=p^{f-1}
                                             a_i+p^{f-2}a_{i+1}+\dots+a_{i-1}\\
                                           &\ge
                                             p^{f-1}(p-1)+\dots+p^{f+i+1-s}(p-1)+p^{f+i-s}p+p^{f+i-1-s}+\dots+1
                        \\&=p^f+p^{f+i-1-s}+\dots+1, \end{align*}which
                      gives the required contradiction.

Having disposed of the case that $m=1$ and $\xi_i=pn_{i-1}$, it
follows from Lemma~\ref{lem: xi_i in terms of n_l} that we may assume
that  $\xi_i=p^m\bigl(n_{i-m}-(p^f-1)\bigr)$. 
We show first that we cannot have~$a_{j-1}\ne p$. Indeed, if this
occurs, then by definition we have $n'_j=n_{j-1}$ and
$\sigma'_j=\sigma_{i-1}$, so that the equations we need to solve are
$i-m=j-1$, and $n_{i-m}-(p^f-1)=n_{j-1}$, which are mutually inconsistent,
since together they imply that
$n_{j-1} - (p^f - 1) = n_{j-1}$.

We are thus reduced to the case when~$a_{j-1}=p$, and,  by the definition of $n'_j$, we
see (since $\sigma'_j=\sigma_{i-m}$) that $i-m$
must be congruent to the greatest integer~$i'$ less than~$j-1$ with
$a_{i'}\ne p-1$. Replacing~$i$ by something congruent to it
modulo~$f$, we may assume that $i-m=i'$, so that $a_{i-m}\ne p-1$, $a_{i-m+1}=\dots =a_{j-2}=p-1$, and
$a_{j-1}=p$. 
Again, we must show that this implies that $j\in\mu(J)$.
By the definition of~$\mu(J)$, this will be the case unless 
$i-m+1,\dots,j-2,j-1,j\notin J$. 
Since we are assuming that
$i\in J$, this implies in particular that $j$ is contained in the interval~$[i-m,i)$.
We now show that this
leads to a contradiction. 
Consider the equation $\xi_i/p^m=n_{i-m}-(p^f-1)$. From the
definitions and the assumptions we are making, we
have \begin{align*}n_{i-m}&=p^{f-1}a_{i-m+1}+\dots+p^{f-x}a_{i-m+x}+\dots+a_{i-m}\\&=p^f+p^{f-m+i-j}a_j+\dots+a_{i-m}, \end{align*}
so that \begin{align*} n_{i-m}-(p^f-1)&=1+
  p^{f-m+i-j}a_j+\dots+a_{i-m}\\&>
  p^{f-m+i-j}+p^{f-m+i-j-1}+\dots+1.\end{align*}
Thus 
\numequation
\label{eqn:preliminary inequality}
	\xi_i = p^m \bigl(n_{i-m}- (p^f -1 )\bigr) > 
  p^{f+i-j}+p^{f+i-j-1}+\dots+p^m.
  \end{equation}
Since
$\xi_i \le p^{2}(p^f-1)/(p-1)$ by Lemma~\ref{lem: bounds on xi},
we conclude that in
particular \[(p^f-1)/(p-1)> \xi_i/p^2 > 
p^{f+i-j - 2} = p^{(f-1)+(i-j-1)},\]which is only possible
if $i=j+1$. Assume now that this is the case.  Then we may rewrite~(\ref{eqn:preliminary inequality}) in the form
\numequation
	\label{eqn:inequality}
	\xi_i = p^m \bigl(n_{i-m}- (p^f -1 )\bigr) > 
  p^{f+1}+p^{f}+\dots+p^m.
  \end{equation}
We also find that
$i-m+1,\dots,i-1\notin J$, so that, from the definition of $\xi_i$ (and
taking into account the fact that $i \in J$), 
we compute
$$\begin{aligned}
\xi_i & =  \ p^{f} r_i+\dots +(-1)^{i-m\notin
J}p^mr_{i-m}-(p^{m-1}r_{i-m+1}+\dots+pr_{i-1}) \\
\le & \ p^{f} r_i+\dots +(-1)^{i-m\notin
J}p^mr_{i-m} \\
\le & \ (p^f + \ldots + p^m)p = 
 p^{f+1} + p^f + \ldots + p^{m+1}. \end{aligned}$$
This contradicts~(\ref{eqn:inequality}), and completes the argument.
\end{proof}

\emergencystretch=3em
\printbibliography
\bigskip
 \end{document}